\documentclass{amsart}
\usepackage{amsmath,amssymb,amsthm,verbatim,color,enumerate,graphicx,caption}

\newtheorem{thm}{Theorem}[section]
 
 \newtheorem{lem}[thm]{Lemma}

\theoremstyle{remark}
\newtheorem{rmk}[thm]{Remark}

\title[Biharmonic Neumann Quantitative Inequality]{On the stability of some isoperimetric inequalities for the fundamental tones of free plates}
\author{D.\ Buoso}
\thanks{Davide Buoso, Dipartimento di Scienze Matematiche ``G.L.\ Lagrange'', Politecnico di Torino, Corso Duca degli Abruzzi, 24, 10129 Torino, Italy. Email: davide.buoso@polito.it (corresponding author).}
\author{L.M.\ Chasman}
\thanks{L. Mercredi Chasman, Division of Science and Mathematics, University of Minnesota -- Morris, 600 East 4th Street, Morris, MN 56267, USA. Email: chasmanm@morris.umn.edu.}
\author{L.\ Provenzano}
\thanks{Luigi Provenzano, Dipartimento di Matematica, Universit\`{a} degli Studi di Padova, Via Trieste,  63, 35126 Padova, Italy. {\it Current affiliation}: \'Ecole Polytechnique F\'ederale de Lausanne, Section de Math\'ematique, Station 8 CH-1015, Lausanne, Switzerland. Email: luigi.provenzano@epfl.ch.}


\begin{document}

\begin{abstract} We provide a quantitative version of the isoperimetric inequality for the fundamental tone of a biharmonic Neumann problem. Such an inequality has been recently established by Chasman adapting Weinberger's argument for the corresponding second order problem. Following a scheme introduced by Brasco and Pratelli for the second order case, we prove that a similar quantitative inequality holds also for the biharmonic operator. We also prove the sharpness of both such an inequality and the corresponding one for the biharmonic Steklov problem.
\end{abstract}

\maketitle

\section{Introduction}

The stability of isoperimetric inequalities is an important question that has gained significant interest in recent decades. For example, the celebrated Faber-Krahn inequality for the smallest eigenvalue of the Dirichlet Laplacian, 
\[
\lambda_1(\Omega)\ge\lambda_1(\Omega^*),
\]
can be improved in the following quantitative form:
\begin{equation}
\lambda_1(\Omega)\ge\lambda_1(\Omega^*)(1+C\mathcal{A}(\Omega)^2),
\label{quantfk}
\end{equation}
for some constant $C>0$. Here $\Omega\subset\mathbb{R}^N$ is a bounded open set, $N\geq2$, $\Omega^*$ is a ball such that $|\Omega|=|\Omega^*|$, and $\mathcal A(\Omega)$ is the so-called Fraenkel asymmetry of the domain $\Omega$ (see~\eqref{fra} for the definition). 
Quantitative versions of the type \eqref{quantfk} have also been established for other isoperimetric inequalities involving eigenvalues of the Laplace operator, see, e.g., \cite{brascosteklov, brasco2015, brascopratelli}.

Fewer isoperimetric inequalities have been established for eigenvalues of the biharmonic operator, namely for the first nontrivial eigenvalue of the Dirichlet (``clamped plate'') problem \cite{ashbaugh,nadirashvili}, of the Neumann (``free plate'') problem \cite{chasmanpreprint,chasman}, and of the Steklov problem introduced in \cite{buosoprovenzano} (see also \cite{buosoprovenzano0}). An isoperimetric inequality is still missing for another Steklov problem introduced in \cite{kuttler68}, the conjectured optimizer being the regular pentagon (see, e.g., \cite{antunesgazzola, bucurgazzola11} and the references therein). 

Among these inequalities, the first one that has been given in quantitative form is the inequality for Steklov problem in \cite{buosoprovenzano}, namely
\begin{equation}
\lambda_2(\Omega)\le\lambda_2(\Omega^*)(1-C\mathcal{A}(\Omega)^2),
\label{quantitative_bp}
\end{equation}
where $\lambda_2(\Omega)$ is the first nontrivial eigenvalue of the biharmonic Steklov problem
\begin{equation}\label{SteklovPDE}
 \begin{cases}\Delta^2u-\tau\Delta u=0 &\text{in $\Omega$,}\\
  \frac{\partial^2 u}{\partial \nu^2}= 0 &\text{on $\partial\Omega$,}\\
		\tau\frac{\partial u}{\partial \nu}
-{\rm div}_{\partial\Omega}\Big(D^2u\cdot \nu\Big)-\frac{\partial\Delta u}{\partial \nu} = \lambda u &\text{on $\partial\Omega$,}
 \end{cases}
\end{equation}
where $\tau$ is a strictly positive constant.

In this paper we provide a quantitative form for the isoperimetric inequality for the first non-trivial eigenvalue of the following biharmonic Neumann problem:
\begin{equation}\label{NeumannPDE}
 \begin{cases}\Delta^2u-\tau\Delta u=\lambda u &\text{in $\Omega$,}\\
  \frac{\partial^2 u}{\partial \nu^2}= 0 &\text{on $\partial\Omega$,}\\
		\tau\frac{\partial u}{\partial \nu}
-{\rm div}_{\partial\Omega}\Big(D^2u\cdot \nu\Big)-\frac{\partial\Delta u}{\partial \nu} = 0 &\text{on $\partial\Omega$.}
 \end{cases}
\end{equation}
We recall that for $N=2$, problem \eqref{NeumannPDE} describes the transverse vibrations of an unconstrained thin elastic plate with shape $\Omega\subset \mathbb{R}^2$ when at rest. The constant $\tau$ represents the ratio of lateral tension to lateral rigidity and is taken to be non-negative.  

When $\tau>0$ and $\Omega\subset\mathbb{R}^N$ is a smooth connected bounded open set, it is known that the spectrum of the Neumann biharmonic operator $\Delta^2-\tau\Delta$ consists entirely of non-negative eigenvalues of finite multiplicity, repeated according to their multiplicity:
\[
0=\lambda_1(\Omega)<\lambda_2(\Omega)\leq\cdots\leq\lambda_j(\Omega)\leq\cdots.
\]
Note that since constant functions satisfy problem \eqref{NeumannPDE} with eigenvalue $\lambda=0$, the first positive eigenvalue is $\lambda_2$, which is usually called the ``fundamental tone'' of the plate. In \cite{chasman}, the author proved that
\begin{equation}\label{iso_neumann}
 \lambda_2(\Omega)\leq \lambda_2(\Omega^*)
\end{equation}
with equality if and only if $\Omega=\Omega^*$. The proof of inequality \eqref{iso_neumann} is based on Weinberger's argument for the Neumann Laplacian, taking suitable extensions of the eigenfunctions of the ball as trial functions (see \cite{weinberger}). In \cite{brascopratelli}, the authors carry out a more careful analysis of such an argument, improving Weinberger's inequality to a quantitative form. In a similar way, we start from the proof of \eqref{iso_neumann} and improve the result to the quantitative inequality~\eqref{quantitative_neumann} by means of this finer analysis. 

The question of sharpness is another important issue that has to be addressed when dealing with quantitative isoperimetric inequalities. More precisely, given an inequality of the form
\begin{equation*}
\lambda_2(\Omega)\leq\lambda_2(\Omega^*)\left(1-\Phi({\rm dist}(\Omega,\mathcal B))\right),
\end{equation*} 
where $\Phi$ is some modulus of continuity, ${\rm dist}(\cdot,\cdot)$ is a suitable distance between open sets and $\mathcal B$ is the family of all balls in $\mathbb R^N$, we say that 
it is sharp if there exists a family $\lbrace\Omega_{\varepsilon}\rbrace_{\varepsilon\in(0,\varepsilon_0)}$ such that ${\rm dist}(\Omega_{\varepsilon},\mathcal B)\rightarrow 0$, $\lambda_2(\Omega_{\varepsilon})\rightarrow\lambda_2(\Omega^*)$ as $\varepsilon\rightarrow 0$, and there exists contants $c_1,c_2>0$ which do not depend on $\varepsilon>0$ and $\Omega^*$ such that
{\begin{equation*}
c_1\Phi({\rm dist}(\Omega_{\varepsilon},\mathcal B))\leq 1-\frac{\lambda_2(\Omega_{\varepsilon})}{\lambda_2(\Omega^*)}\leq c_2\Phi({\rm dist}(\Omega_{\varepsilon},\mathcal B)),
\end{equation*}}
as $\varepsilon\rightarrow 0$. Note that, in our case, the distance function is given by the Fraenkel asymmetry ${\rm dist}(\Omega,\mathcal B)=\mathcal A(\Omega)$ while the modulus of continuity is $\Phi(t)=Kt^2$, for some $K>0$. By means of the construction introduced in \cite{brascosteklov,brascopratelli}, we prove in Section~\ref{sharpness_neumann} that the quantitative Neumann inequality \eqref{quantitative_neumann} is sharp. 

It is worth noting that in the Neumann Laplacian case in \cite{brascopratelli}, the authors try, as a first guess, to consider ellipsoids as the family $\lbrace\Omega_{\varepsilon}\rbrace_{\varepsilon\in(0,\varepsilon_0)}$, with the ball $\Omega_0$ being the maximizer. Unfortunately, this is not a good family to prove sharpness; this can be explained observing that different directions of perturbation behave in a different way with respect to the fundamental tone. In particular, some directions are not ``good enough'' to see the sharpness (cf.\ \cite[Remark 5.2]{brascopratelli}). This phenomenon can be observed in our case as well: therefore we need to restrict our analysis by excluding some directions. See (\ref{perturbation}) and Remark \ref{directions}. 

The Steklov problem \eqref{SteklovPDE} is of particular interest despite its recent introduction, since in \cite{buosoprovenzano} the authors show that it has a very strict relationship with the Neumann problem \eqref{NeumannPDE}. Using a mass perturbation argument, they prove that the Steklov problem can in fact be viewed as a limiting Neumann problem where the mass is distributed only on the boundary. Note that this construction was already performed in \cite{lambertiprozisaac} for the Laplace operator, obtaining similar results (see also \cite{dallarivaproz,lambertiproz} for the computation of the topological derivative). Moreover, this justifies the fact of thinking of Steklov problems in terms of vibrating objects (plates or membranes) where the mass lies only on the boundary (see \cite{steklov}). The authors also prove the quantitative inequality \eqref{quantitative_bp} by adapting an argument due to Brock (see \cite{brock}) for the Steklov Laplacian to the biharmonic case in the refined version of \cite{brascosteklov}. However, they do not discuss its sharpness. The similarity of the variational characterization of Neumann and Steklov eigenvalues allows us to prove that inequality \eqref{quantitative_bp} is sharp by an easy adaptation of the arguments used in the Neumann case.

The paper is organized as follows. In Section~\ref{preliminaries}, we give some preliminary results and introduce the notation. Section~\ref{proof_neumann_quantitative} is devoted to the Neumann quantitative isoperimetric inequality \eqref{quantitative_neumann}, the sharpness of which we prove in Section~\ref{sharpness_neumann}. Finally, in Section~\ref{sharpness_steklov} we prove that the Steklov inequality \eqref{quantitative_bp} is sharp.

\section{Preliminaries and notation}\label{preliminaries}
We introduce here the notation used throughout the paper and recall some fundamental results proved in \cite{chasman}.

Let $B$ be the unit ball in $\mathbb{R}^N$ centered at the origin and $\omega_N$ be the Lebesgue measure $|B|$ of $B$.  

We denote by $j_1$ and $i_1$ the ultraspherical and modified ultraspherical Bessel functions of the first kind and order $1$ respectively. They can be expressed in terms of standard Bessel and modified Bessel functions of the first kind $J_{\nu}, I_{\nu}$ as follows:
\[
j_1(z)=z^{1-N/2}J_{N/2}(z),\qquad i_1(z)=z^{1-N/2}I_{N/2}(z).
\]
For more information on Bessel and modified Bessel functions, see, e.g., \cite[\S 9]{abram}.

We will define trial functions in terms of the eigenfunctions corresponding to $\lambda_2(B)$ of the Neumann problem. For a fixed $\tau>0$, we take positive constants $a,b$ satisfying $a^2b^2=\lambda_2(B)$ and $b^2-a^2=\tau$. We set
\[
R(r)=j_1(ar)+\gamma i_1(br),\qquad\text{where}\qquad \gamma=-\frac{a^2 j_1''(a)}{b^2 i_1''(b)}.
\]
We then define the function $\rho:[0,+\infty)\to[0,+\infty)$ as
\begin{equation*}
\rho(r)=\begin{cases}
R(r),&r\in[0,1)\\
R(1)+(r-1)R'(1),&r\in[1,+\infty).
\end{cases}
\end{equation*}

Let $u_k:\mathbb{R}^N\to\mathbb{R}$ be defined by
\begin{equation}\label{uk}
u_k(x):=\rho(|x|)\frac{x_k}{|x|},
\end{equation}
for $k=1,\dots,N$. The functions ${u_k}_{|_{B}}$ are in fact the eigenfunctions associated with the eigenvalue $\lambda_2(B)$ of the Neumann problem \eqref{NeumannPDE} on the unit ball $B$. Recall that $\lambda_2(B)$ has multiplicity $N$ (see \cite[Theorem 3]{chasman11}). Moreover, we have (see \cite[p.\ 437]{chasman})
\begin{align*}
\sum_{k=1}^N |u_k|^2&=\rho(|x|)^2,\\
\sum_{k=1}^N|D u_k|^2&=\frac{N-1}{|x|^2}\rho(|x|)^2+(\rho'(|x|))^2,\\
\sum_{k=1}^N|D^2 u_k|^2&=(\rho''(|x|))^2+\frac{3(N-1)}{|x|^4}(\rho(|x|)-|x|\rho'(|x|))^2.
\end{align*}
We denote by $N[\rho]$ the quantity
\[
N[\rho]:=\sum_{k=1}^N  |D^2u_k|^2+\tau|D u_k|^2.
\]

We recall some properties enjoyed by the functions $\rho$ and $N[\rho]$ which were proved in \cite{chasman}.
\begin{lem}\label{pro}
The function $\rho$ satisfies the following properties.
\begin{enumerate}[i)]
\item $\rho''(r)\leq 0$ for all $r\geq 0$, therefore $\rho'$ is non-increasing.
\item $\rho(r)-r\rho'(r)\geq 0$, equality holding only for $r=0$.
\item The function $\rho(r)^2$ is strictly increasing.
\item The function ${\rho(r)^2}/{r^2}$ is decreasing.
\item The function ${3(\rho(r)-r\rho'(r))^2}/{r^4}+\tau{\rho^2(r)}/{r^2}$ is decreasing.
\item $N[\rho(r_1)]>N[\rho(r_2)]$ for any $r_1\in [0,1)$, $r_2\in [1,+\infty)$.
\item For all $r\geq 0$ we have
\[
 N[\rho(r)]=(\rho''(r))^2+\frac{3(N-1)(\rho(r)-r\rho'(r))^2}{r^4}+\tau(N-1)\frac{\rho^2(r)}{r^2}+\tau(\rho'(r))^2.
\]
\item For all $r\geq 1$, $N[\rho(r)]$ is decreasing.
\end{enumerate}
\end{lem}

To conclude this section, let us recall the definition of the Fraenkel asymmetry $\mathcal A(\Omega)$ of a set $\Omega\subset\mathbb{R}^N$:
\begin{equation}\label{fra}
\mathcal A(\Omega):=\inf\left\{\frac{|\Omega\triangle {B}|}{|\Omega|}: {B}\text{\ is a ball such that}\ |{B}|=|\Omega|\right\}.
\end{equation}


\section{Quantitative isoperimetric inequality for the Neumann problem}\label{proof_neumann_quantitative}

In this section we state and prove the {quantitative isoperimetric inequality for the fundamental tone of the Neumann problem \eqref{NeumannPDE}}.

\begin{thm}\label{NeumannQI}
For every bounded domain $\Omega$ in $\mathbb{R}^N$ of class $C^1$ the following estimate holds
\begin{equation}\label{quantitative_neumann}
\lambda_2(\Omega)\leq\lambda_2(\Omega^*)\left(1-\eta_{N,\tau,|\Omega|}\mathcal A(\Omega)^2\right),
\end{equation}
where $\eta_{N,\tau,|\Omega|}>0$, $\Omega^*$ is a ball such that $|\Omega^*|=|\Omega|$, and $\lambda_2(\Omega)$, $\lambda_2(\Omega^*)$ are the first positive eigenvalues of problem \eqref{NeumannPDE} on $\Omega$, $\Omega^*$ respectively.
\end{thm}

\begin{proof}
Let $\Omega$ be a bounded domain in $\mathbb{R}^N$ of class $C^1$ with the same measure as the unit ball $B$. We recall the variational characterization of the second eigenvalue $\lambda_2(\Omega)$ of \eqref{NeumannPDE} on $\Omega$:
\begin{equation}\label{ray}
\lambda_2(\Omega)=\inf_{\substack{0\ne u\in H^2(\Omega)\\ \int_{\Omega}u dx=0}}\frac{\int_{\Omega}|D^2u|^2+\tau|Du|^2dx}{\int_{\Omega}u^2dx}.
\end{equation}
Let $u_k(x)$, for $k=1,\dots,N$, be the eigenfunctions corresponding to $\lambda_2(B)$ defined in \eqref{uk}. Clearly ${u_k}_{|_{\Omega}}\in H^2(\Omega)$ by construction. It is possible to choose the origin of the coordinate axes in $\mathbb{R}^N$ in such a way that $\int_{\Omega}u_k dx=0$ for all $k=1,\dots,N$. With this choice, the functions $u_k$ are suitable trial functions for the Rayleigh quotient \eqref{ray}. Once we have fixed the origin, let
$$
\alpha:=\frac{|\Omega\triangle B|}{|\Omega|}.
$$
By definition of Fraenkel asymmetry, we have
\begin{equation}\label{A}
\mathcal A(\Omega)\leq\alpha\leq 2. 
\end{equation}
From the variational characterization \eqref{ray}, it follows that for each $k=1,\dots,N$,
\[
\lambda_2(\Omega)\leq\frac{\int_{\Omega}|D^2u_k|^2+\tau|D u_k|^2dx}{\int_{\Omega}u_k^2 dx}.
\]
We multiply both sides by $\int_{\Omega}u_k^2 dx$ and sum over $k=1,\dots,N$, obtaining
\begin{equation}\label{n1}
\lambda_2(\Omega)\leq\frac{\int_{\Omega}N[\rho]dx}{\int_{\Omega}\rho^2 dx}.
\end{equation}
The same procedure for $\lambda_2(B)$ clearly yields
\begin{equation}\label{b1}
\lambda_2(B)=\frac{\int_{B}N[\rho]dx}{\int_{B}\rho^2 dx}.
\end{equation}
From \eqref{n1} and \eqref{b1}, it follows that
\begin{equation}\label{ineq1}
\lambda_2(B)\int_B \rho^2 dx-\lambda_2(\Omega)\int_{\Omega}\rho^2 dx\geq\int_B N[\rho]dx-\int_{\Omega}N[\rho]dx\geq 0,
\end{equation}
where the last inequality follows from Lemma~\ref{pro}, {\it iv)} and \cite[Lemma 14]{chasman}. 

Now we consider the two balls $B_1$ and $B_2$ centered at the origin with radii $r_1,r_2$ taken such that $|\Omega\cap B|=|B_1|=\omega_N r_1^N$ and $|\Omega\setminus B|=|B_2\setminus B|=\omega_N(r_2^N-1)$. Then $|B_2|=\omega_N r_2^N$, and by construction
\begin{equation}\label{aster}
1-r_1^N=\frac{\alpha}{2}=r_2^N-1.
\end{equation}
This is due to the fact that $|\Omega|+|B|=|\Omega\triangle B|+2 |\Omega\cap B|$, and then $1-r_1^N=\alpha/2$. Similarly, $|\Omega\setminus B|+|\Omega\cap B|=|\Omega|$, hence $r_1^N=2-r_2^N$, and then $r_2^N-1=\alpha/2$.

Now we observe, again by Lemma~\ref{pro}, {\it vi)} and {\it viii)}, that
\[
\int_{\Omega}N[\rho]dx\leq\int_{B_1}N[\rho]dx+\int_{B_2\setminus B}N[\rho]dx.
\]
From this and \eqref{ineq1}, we obtain
\begin{align}\label{ineq2}
\lambda_2(B)\int_B\rho^2 dx-\lambda_2(\Omega)\int_{\Omega}\rho^2 dx&\geq\int_B N[\rho]dx-\int_{\Omega}N[\rho]dx\\
&\geq\int_{B\setminus B_1}N[\rho]dx-\int_{B_2\setminus B}N[\rho]dx.\nonumber
\end{align}
Since the function $\rho(r)^2$ is strictly increasing by Lemma~\ref{pro}, {\it iii)}, we have
\begin{equation*}
\int_{\Omega}\rho^2 dx\geq\int_B\rho^2 dx=N\omega_N\int_0^1\rho^2(r)r^{N-1}dr=:C^{(1)}_{N,\tau},
\end{equation*}
hence,
\begin{align}\label{passoA}
\lambda_2(B)&\int_B\rho^2 dx-\lambda_2(\Omega)\int_{\Omega}\rho^2 dx\\
&\leq\left(\lambda_2(B)-\lambda_2(\Omega)\right)\int_B\rho^2 dx +\lambda_2(\Omega)\left(\int_B\rho^2 dx-\int_{\Omega}\rho^2 dx\right)\nonumber\\
&\leq C^{(1)}_{N,\tau}\left(\lambda_2(B)-\lambda_2(\Omega)\right).\nonumber
\end{align}

Now we consider the right-hand side of \eqref{ineq2}. We write $N[\rho]$ more explicitly in terms of $\rho$, obtaining:
\begin{align}\label{passoB1}
\int_{B\setminus B_1}&N[\rho]dx=N\omega_N\int_{r_1}^1\Big((\rho''(r))^2+\frac{3(N-1)(\rho(r)-r\rho'(r))^2}{r^4}\\
&\qquad\qquad\qquad\qquad+\tau(\rho'(r))^2+\frac{\tau (N-1)}{r^2}\rho(r)^2\Big)r^{N-1}dr\nonumber\\
&\geq N\omega_N\int_{r_1}^1\left(\frac{3(N-1)(\rho(r)-r\rho'(r))^2}{r^4}+\tau(\rho'(r))^2+\frac{\tau (N-1)}{r^2}\rho(r)^2\right)r^{N-1}dr\nonumber\\
&\geq \omega_N\left(3(N-1)(R(1)-R'(1))^2+\tau R'(1)^2+\tau (N-1)R(1)^2\right)(1-r_1^N),\nonumber
\end{align}
where in the last inequality, we used the fact that $N[\rho]-(\rho'')^2$ is non-increasing in $r$ (see Lemma~\ref{pro}, {\it i)} and {\it v)}).
Moreover,
\begin{align}\label{passoB2}
\int_{B_2\setminus B}&N[\rho]dx\\
&=N\omega_N\int_1^{r_2}\left(\frac{3(N-1)}{r^4}(R(1)-R'(1))^2+\tau R'(1)^2\right.\nonumber\\
                &\qquad\qquad+\frac{\tau(N-1)}{r^2}\Big((R(1)-R'(1))^2+2rR'(1)(R(1)-R'(1))\Big)\nonumber\\
                 &\qquad\qquad\left.+\frac{\tau(N-1)}{r^2}\Big(r^2R'(1)^2\Big)\right)r^{N-1}dr\nonumber\\
&\leq N\omega_N\int_1^{r_2}\left(N\tau R'(1)^2+\frac{N-1}{r}\left((3+\tau)(R(1)-R'(1))^2\right.\right.\nonumber\\
               &\qquad\qquad\left.+2\tau R'(1)(R(1)-R'(1))\right)\Big)r^{N-1}dr\nonumber\\
&=N\omega_N\tau R'(1)^2(r_2^N-1)+N\omega_N\left((3+\tau)(R(1)-R'(1))^2\right.\nonumber\\
                 &\qquad\left.+2\tau R'(1)(R(1)-R'(1))\right)(r_2^{N-1}-1),\nonumber
\end{align}
\normalsize
where we have estimated the quantities ${1}/{r^2}$ and ${1}/{r^4}$ by ${1}/{r}$. We note that $r_2=\left(1+{\alpha}/{2}\right)^{{1}/{N}}$ and $0\leq\alpha\leq 2$. Using the Taylor expansion up to order $1$ and remainder in Lagrange form, we obtain
\begin{align}\label{1n}
r_2^{N-1}&=1+\frac{N-1}{N}\frac{\alpha}{2}-\frac{(N-1)\left(1+\frac{\xi}{2}\right)^{\frac{N-1}{N}-2}}{8N^2}\alpha^2\\
&\leq 1+\frac{N-1}{N}\frac{\alpha}{2}-\frac{(N-1)2^{\frac{N-1}{N}-2}}{8N^2}\alpha^2=1+\frac{N-1}{N}\frac{\alpha}{2}-c_{N}\alpha^2,\nonumber
\end{align}
for some $\xi\in(0,\alpha)$, where $c_{N}$ is a positive constant which depends only on $N$. Using \eqref{aster}, \eqref{passoB1}, \eqref{passoB2}, and \eqref{1n}, in the right-hand side of \eqref{ineq2}, we obtain:
\begin{align}\label{tofinal1}
\int_{B\setminus B_1}&N[\rho]dx-\int_{B_2\setminus B}N[\rho]dx \\
&\ge-N\omega_N\Big((3+\tau)(R(1)-R'(1))^2+2\tau R'(1)(R(1)-R'(1))\Big)\left(\frac{N-1}{N}\frac{\alpha}{2}- c_{N}\alpha^2\right)\nonumber\\
&\qquad+\omega_N\left(3(N-1)(R(1)-R'(1))^2+\tau R'(1)^2+\tau (N-1) R(1)^2\right)\frac{\alpha}{2}\nonumber\\
&\qquad -N\omega_N\tau R'(1)^2\frac{\alpha}{2}\nonumber\\
&=:C^{(2)}_{N,\tau}\alpha^2,\nonumber
\end{align}
where the constant $C^{(2)}_{N,\tau}>0$ is given by
\[
C^{(2)}_{N,\tau}=N\omega_N\left((3+\tau)(R(1)-R'(1))^2+2\tau R'(1)(R(1)-R'(1))\right)c_{N}.
\]
\normalsize
From \eqref{A}, \eqref{ineq2}, \eqref{passoA}, and \eqref{tofinal1}, it follows that
\[
\lambda_2(B)-\lambda_2(\Omega)\geq\frac{C^{(2)}_{N,\tau}}{C^{(1)}_{N,\tau}}\mathcal A(\Omega)^2,
\]
and therefore,
\begin{equation}\label{quantN-1}
\lambda_2(\Omega)\leq\lambda_2(B)\left(1-\frac{C_{N,\tau}^{(2)}}{\lambda_2(B)C_{N,\tau}^{(1)}}\mathcal A(\Omega)^2\right).
\end{equation}
The isoperimetric inequality is thus proved in the case of $\Omega$ with the same measure as the unit ball. The inequality for a generic domain $\Omega$ follows from scaling properties of the eigenvalues of problem \eqref{NeumannPDE}. Writing our eigenvalues as $\lambda_2(\tau,\Omega)$ to make explicit the dependence on the parameter $\tau$, we have
\begin{equation}\label{scaling}
\lambda_2(\tau,\Omega)=s^4\lambda_2(s^{-2}\tau,s\Omega),
\end{equation}
for all $s>0$. From \eqref{quantN-1} and taking $s=(\omega_N/|\Omega|)^{1/N}$ in \eqref{scaling}, it follows that for every $\Omega$ in $\mathbb R^N$ of class $C^1$ we have
\begin{align*}
\lambda_2(\tau,\Omega)&=s^4\lambda_2(s^{-2}\tau,s\Omega)\\
&\leq s^4\lambda_2(s^{-2}\tau,B)\left(1-\frac{C^{(2)}_{N,s^{-2}\tau}}{\lambda_2(s^{-2}\tau,B)C^{(1)}_{N,s^{-2}\tau}}\mathcal A(s\Omega)\right)\\
&=\lambda_2(\tau,\Omega^*)\left(1-\frac{C^{(2)}_{N,s^{-2}\tau}}{\lambda_2(s^{-2}\tau,B)C^{(1)}_{N,s^{-2}\tau}}\mathcal A(\Omega)\right).
\end{align*}
We set 
$$\eta_{N,\tau,|\Omega|}:=\frac{C^{(2)}_{N,s^{-2}\tau}}{\lambda_2(s^{-2}\tau,B)C^{(1)}_{N,s^{-2}\tau}}.$$
This concludes the proof of the theorem.
\end{proof}


\begin{rmk}  One generalization of the {biharmonic Neumann problem \eqref{NeumannPDE}} is to consider the case where the plate is made of a material with a nonzero Poisson's ratio $\sigma$, which replaces the term $|D^2u|^2$ in the Rayleigh quotient by $(1-\sigma)|D^2u|^2+\sigma(\Delta u)^2$. A partial result towards the non-quantitative form of the isopermetric inequality has been obtained for certain values of $\tau>0$ and $\sigma\in(-1/(N-1),1)$, proved by the second author in \cite{chasmanpreprint} (see also \cite{buoso15,prozkalamata}). In this case, the proof of Theorem \ref{NeumannQI} can be easily adapted, yielding
 \[
  \lambda_2(B)-\lambda_2(\Omega)\geq \frac{C^{(3)}_{N,\tau}}{C^{(1)}_{N,\tau}}\mathcal{A}(\Omega)+\frac{C^{(2)}_{N,\tau}}{C^{(1)}_{N,\tau}}\mathcal{A}(\Omega)^2,
 \]
 where $C^{(1)}_{N,\tau}$, $C^{(2)}_{N,\tau}$ are as in the proof of Theorem \ref{NeumannQI}, and
 \[
  C^{(3)}_{N,\tau}=\frac{1}{2}(R(1)-R'(1))^2(N-1)\sigma\Big(\sigma(N-1)(\sigma-2)+N-2\Big).
 \]
 This result is not particularly satisfying, since it carries all of the same limitations of the non-quantitative result (only being valid for certain $\tau$ and $\sigma$), and in some cases it is strictly worse, since $C^{(3)}_{N,\tau}$ is non-negative only when $0\leq \sigma\leq 1-1/\sqrt{N-1}$. 
\end{rmk}

\begin{rmk}
Even though we are able to give a quantitative isoperimetric inequality for the fundamental tone of problem \eqref{NeumannPDE}, very little is known in this regard for higher eigenvalues. To the best of our knowledge, only criticality results are available in the literature, where the ball is shown to be a critical domain under volume constraint (see, e.g., \cite{buoso15,buosolamberti15,buosoprovenzano}). However, as in the second-order case, the ball is not expected to be an optimizer for higher eigenvalues.
\end{rmk}

\section{Sharpness of the Neumann inequality}\label{sharpness_neumann}

In this section, we prove the sharpness of inequality \eqref{quantitative_neumann}.

\begin{thm}\label{theorem2}
Let $B$ be the unit ball in $\mathbb{R}^N$ centered at zero. There exist a family $\left\{\Omega_{\epsilon}\right\}_{\epsilon>0}$ of smooth domains and positive constants $c_1,c_2,c_3,c_4$ and $r_1, r_2,r_3,r_4$ independent of $\epsilon>0$ such that
\begin{equation}
\label{sharp1}
r_1\epsilon^2\le\Big||\Omega_{\epsilon}|-|B|\Big|\leq r_2\epsilon^2,
\end{equation}
\begin{equation}
\label{sharp2}
c_1\epsilon\leq c_2\mathcal A(\Omega_{\epsilon})\leq\frac{|\Omega_{\epsilon}\triangle B|}{|\Omega_{\epsilon}|}\leq c_3\mathcal A(\Omega_{\epsilon})\leq c_4\epsilon,
\end{equation}
and
\begin{equation}
\label{sharp3}
r_3\epsilon^2\le\left|\lambda_2(\Omega_{\epsilon})-\lambda_2(B)\right|\leq r_4\epsilon^2,
\end{equation}
for all $\epsilon\in(0,\epsilon_0)$, where $\epsilon_0>0$ is sufficiently small, {and $\lambda_2(\Omega_\epsilon)$, $\lambda_2(B)$ are the first positive eigenvalues of problem \eqref{NeumannPDE} on $\Omega_\epsilon$, $B$ respectively}.
\end{thm}

In order to prove Theorem \ref{theorem2}, we start by defining the family of domains $\left\{\Omega_{\epsilon}\right\}_{\epsilon>0}$ as follows (see Figure \ref{figura}):
\begin{equation}\label{family}
\Omega_{\epsilon}=\left\{x\in\mathbb R^N: x=0 {\text\ or\ }|x|<1+\epsilon\psi\left(\frac{x}{|x|}\right)\right\},
\end{equation}
where $\psi$ is a function belonging to the following class:
\begin{equation}\label{perturbation}
\mathcal{P}=\left\{\psi\in C^{\infty}(\partial B):\int_{\partial B}\psi d\sigma=\int_{\partial B}(a\cdot x)\psi d\sigma=\int_{\partial B}(a\cdot x)^2\psi d\sigma=0,\ \forall a\in\mathbb{R}^N\right\}.
\end{equation}
\begin{figure}[hb]
    \centering
   
        \includegraphics[width=\textwidth]{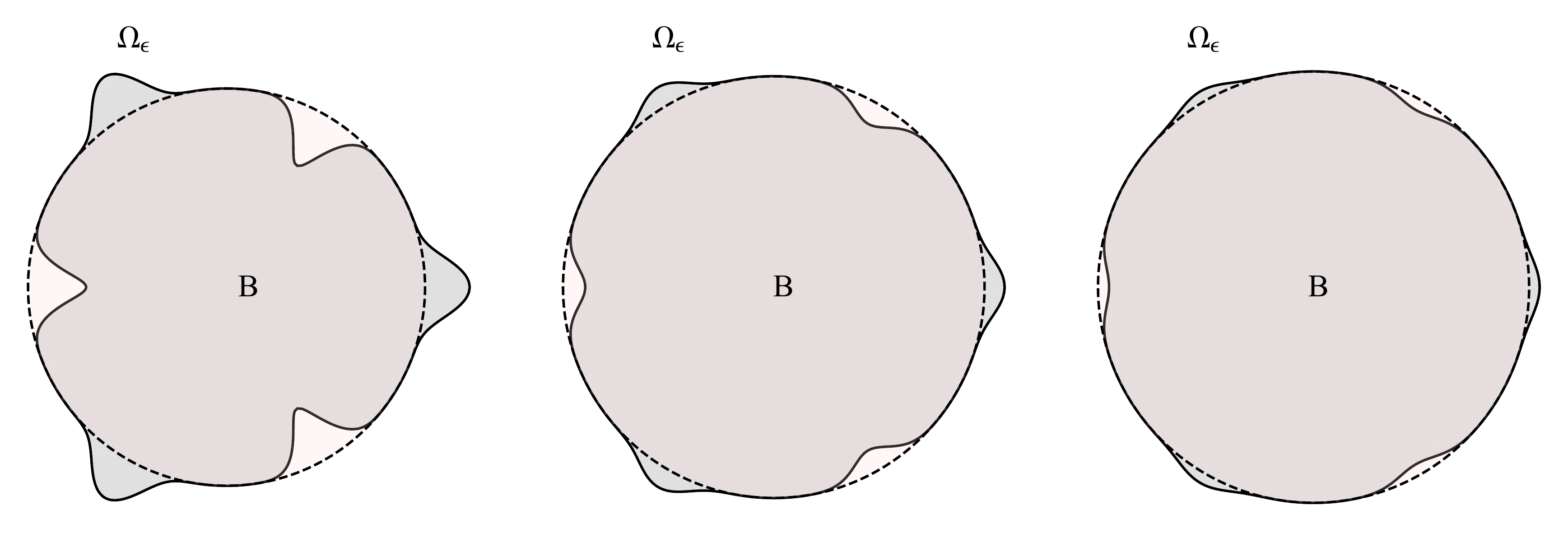}
        \caption{Domains $\Omega_{\varepsilon}$ defined by \eqref{family} with a given $\psi\in\mathcal P$.}
        \label{figura}

    \end{figure}

Under our choice of $\Omega_\epsilon$, the existence of constants $r_1,r_2,c_1,\dots,c_4$ satisfying inequalities \eqref{sharp1} and \eqref{sharp2} follow immediately from \cite[Lemma 6.2]{brascosteklov}. Thus, we need only prove \eqref{sharp3}.

Let $\lambda_2(\Omega_{\epsilon})$ be the first positive eigenvalue of the Neumann problem \eqref{NeumannPDE} on $\Omega_{\epsilon}$, and let $u_{\epsilon}$ be an associated eigenfunction normalized by $\|u_{\epsilon}\|_{L^2(\Omega_{\epsilon})}=1$,
so that
\[
\int_{\Omega_{\epsilon}}|D^2u_{\epsilon}|^2+\tau|Du_{\epsilon}|^2\,dx=\lambda_2(\Omega_{\epsilon}).
\]
By standard elliptic regularity (see e.g., \cite[\S 2.4.3]{gazzola}), since $\Omega_{\epsilon}$ is of class $C^{\infty}$ by construction, we may take a sufficiently small $\epsilon_0>0$ so that $u_{\epsilon}  \in C^{\infty}(\overline\Omega_{\epsilon})$ for all $\epsilon\in(0,\epsilon_0)$. Moreover, for all $k\in\mathbb N$, the sets $\Omega_{\epsilon}$ are of class $C^k$ uniformly in $\epsilon\in(0,\epsilon_0)$, which means that there exist constants $H_k>0$ independent of $\epsilon$ that satisfy
\begin{equation}\label{regularity}
\|u_{\epsilon}  \|_{C^k(\overline{\Omega_{\epsilon}})}\leq H_k.
\end{equation} 
Now let $\tilde{u}_{\epsilon} $ be a $C^4$ extension of $u_{\epsilon}$ to some open neighborhood $A$ of $B\cup\Omega_{\epsilon}$. Then, there exists $K_A>0$ independent of $\epsilon>0$ for which
\begin{equation}\label{regularity_extension}
\|\tilde{u}_{\epsilon} \|_{C^4(\overline A)}\leq K_A\|u_{\epsilon}  \|_{C^4(\overline{\Omega_{\epsilon}})}\leq K_A H_4.
\end{equation}
From the fact that $\int_{\Omega_{\epsilon}}u_{\epsilon} \,dx=0$ and $|B\setminus\Omega_{\epsilon}|,|\Omega_{\epsilon}\setminus B|\in O(\epsilon)$ as $\epsilon\rightarrow 0$, it follows that the quantity $\delta:=\frac{1}{|B|}\int_B\tilde{u}_{\epsilon}\,dx$ satisfies
\begin{equation}\label{delta_bound}
\delta=\frac{1}{|B|}\int_B\tilde{u}_{\epsilon} \,dx=\frac{1}{|B|}\left(\int_{B\setminus \Omega_{\epsilon}}\tilde{u}_{\epsilon}\,dx-\int_{\Omega_{\epsilon}\setminus B}u_{\epsilon} \,dx\right)\leq c\epsilon,
\end{equation}
where $c>0$ does not depend on $\epsilon\in(0,\epsilon_0)$. Now let us set 
\begin{equation}\label{test}
v_{\epsilon} :={\tilde u_{\epsilon|_B}}-\delta.
\end{equation}
The function $v_{\epsilon}$ is of class $C^4(\overline B)$ with $\int_B v_{\epsilon} \,dx=0$ and
\begin{equation}\label{regularity_v}
\|v_{\epsilon}  \|_{C^4(\overline B)}\leq K_1
\end{equation}
for some constant $K_1>0$ independent of $\epsilon\in(0,\epsilon_0)$. Therefore, $v_{\epsilon}$ is a suitable trial function for the Rayleigh quotient of $\lambda_2(B)$ (see formula \eqref{ray}). Thus,
\begin{equation}\label{minmax_1}
\lambda_2(B)\leq\frac{\int_B |D^2 v_{\epsilon}|^2+\tau|Dv_{\epsilon}|^2\,dx}{\int_B {v_{\epsilon}  }^2\,dx}.
\end{equation}
We now consider the quantity $\left|\int_B v_{\epsilon}  ^2-\tilde{u}_{\epsilon} ^2\,dx\right|$. We have
\begin{align}\label{ineq_1}
\left|\int_B v_{\epsilon}  ^2-\tilde{u}_{\epsilon} ^2\,dx\right|&=\left|\int_B \delta^2-2\delta\tilde{u}_{\epsilon}\,dx\right|=\left|\int_B\delta(v_{\epsilon}-\tilde{u}_{\epsilon})\,dx\right|\\
&=\frac{1}{|B|}\left(\int_B \tilde{u}_{\epsilon}\,dx \right)^2\leq K_2\epsilon^2,\nonumber
\end{align}
where $K_2>0$ is a positive constant independent of $\epsilon\in(0,\epsilon_0)$. Moreover, by \eqref{regularity_extension} and \eqref{regularity_v}, we have that 
\begin{align}\label{ineq_2}
\left|\int_{B\setminus\Omega_{\epsilon}}v_{\epsilon}  ^2-\tilde{u}_{\epsilon} ^2dx\right|&\leq \int_{B\setminus\Omega_{\epsilon}}|v_{\epsilon}  ^2-\tilde{u}_{\epsilon} ^2|dx\leq K_3\int_{B\setminus\Omega_{\epsilon}}|v_{\epsilon}  -\tilde{u}_{\epsilon} |dx\\
&=K_3\frac{|B\setminus\Omega_{\epsilon}|}{|B|}\left|\int_B\tilde{u}_{\epsilon}\,dx\right|\leq K_4\epsilon^2,\nonumber
\end{align}
where $K_3,K_4>0$ are positive constants independent of $\epsilon\in(0,\epsilon_0)$. Therefore, from \eqref{minmax_1}, \eqref{ineq_1}, and \eqref{ineq_2}, it follows that
\begin{align}\label{estimate_1}
\lambda_2(B)&\leq\frac{\int_{B\cap\Omega_{\epsilon}}|D^2u_{\epsilon}|^2+\tau|Du_{\epsilon}|^2dx+\int_{B\setminus\Omega_{\epsilon}}|D^2 v_{\epsilon}|^2+\tau|D v_{\epsilon}|^2\,dx}{\int_B\tilde{u}_{\epsilon} ^2dx-K_2\epsilon^2}\\
&\le\frac{\lambda_2(\Omega_{\epsilon})+\int_{B\setminus\Omega_{\epsilon}}|D^2v_{\epsilon}|^2+\tau|Dv_{\epsilon}|^2dx-\int_{\Omega_{\epsilon}\setminus B}|D^2 u_{\epsilon}|^2+\tau|Du_{\epsilon}|^2\,dx}{1+\int_{B\setminus\Omega_{\epsilon}}v_{\epsilon}  ^2dx-\int_{\Omega_{\epsilon}\setminus B}u_{\epsilon}^2dx-(K_2+K_4)\epsilon^2}.\nonumber
\end{align}
We introduce now the two error terms $R_1(\epsilon)$ and $R_2(\epsilon)$ defined by
\[
R_1(\epsilon):=\int_{B\setminus\Omega_{\epsilon}}|D^2v_{\epsilon}|^2+\tau|Dv_{\epsilon}|^2dx-\int_{\Omega_{\epsilon}\setminus B}|D^2u_{\epsilon}|^2+\tau|Du_{\epsilon}|^2\,dx
\]
and
\[
R_2(\epsilon):=\int_{B\setminus\Omega_{\epsilon}}v_{\epsilon}  ^2dx-\int_{\Omega_{\epsilon}\setminus B}u_{\epsilon}^2dx.
\]
Then inequality \eqref{estimate_1} can be rewritten as
\begin{equation}\label{estimate_2}
\lambda_2(B)\leq\frac{\lambda_2(\Omega_{\epsilon})+R_1(\epsilon)}{1+R_2(\epsilon)-K_5\epsilon^2}.
\end{equation}
From the uniform estimates \eqref{regularity} and \eqref{regularity_v} on $u_{\epsilon}$ and $v_{\epsilon}$, it easily follows that $R_1,R_2\in O(\epsilon)$ as $\epsilon\rightarrow 0$, which together with \eqref{estimate_2} immediately yields $\lambda_2(B)\leq\lambda_2(\Omega_{\epsilon})+C\epsilon$ for some constant $C>0$ independent of $\epsilon\in(0,\epsilon_0)$ (taking $\epsilon_0>0$ smaller if necessary).

We observe that, due to the strict relation of $R_1(\epsilon)$ and $R_2(\epsilon)$ with the difference $\lambda_2(B)-\lambda_2(\Omega_{\epsilon})$, a better estimate for $R_1(\epsilon)$ and $R_2(\epsilon)$ provides a better estimate for $\lambda_2(B)-\lambda_2(\Omega_{\epsilon})$. More precisely, we have the following
\begin{lem}\label{lemma_refinement}
Let $\omega:[0,1]\rightarrow[0,+\infty)$ be a continuous function such that $t^2/K\leq\omega(t)\leq K t$, for some $K>0$. If there exists a constant $C>0$ such that $|R_1(\epsilon)|$, $|R_2(\epsilon)|\leq C\omega(\epsilon)$, then there exists a constant $C'>0$ such that
\[
\lambda_2(B)\leq\lambda_2(\Omega_{\epsilon})+C'\omega(\epsilon)
\]
for every sufficient small $\epsilon>0$.
\end{lem}
\begin{proof}
We refer to \cite[Lemma 6.2]{brascopratelli} for the proof (see also \cite[Lemma 6.7]{brascosteklov}).
\end{proof}

We also need the following
\begin{lem}\label{lemma2}
Let $\omega$ be a function as in Lemma \ref{lemma_refinement}, and let $v_{\epsilon}$ be as in \eqref{test}. Suppose that there exists $C>0$ such that for all $\epsilon>0$ sufficiently small we have $|R_1(\epsilon)|, |R_2(\epsilon)|\leq C\omega(\epsilon)$. Then there exists an eigenfunction $\xi_{\epsilon}$ associated with $\lambda_2(B)$ such that
\[
\|v_{\epsilon}  -\xi_{\epsilon}\|_{C^3(\overline B)}\leq\tilde C\sqrt{\omega(\epsilon)}
\]
for some $\tilde C>0$ independent of $\epsilon>0$.
\begin{proof} Take $\{\xi_n\}_{n\geq 1}$ to be an orthonormal basis of $L^2(B)$ consisting of eigenfunctions of problem \eqref{NeumannPDE} on the unit ball $B$. Note that from such a normalization, we have
\[
\int_B|D^2\xi_n|^2+\tau|D\xi_n|^2\,dx=\lambda_n(B)\,\quad\forall n\in\mathbb{N}.
\]

We may write $v_{\epsilon} =\sum_{n=1}^{+\infty}a_n(\epsilon)\xi_n$. Note that $a_1(\epsilon)\equiv 0$, since $v_{\epsilon}$ has zero integral mean over $B$ and $\xi_1$ is a constant. We have
\begin{align*}
\sum_{n=2}^{+\infty}a_n(\epsilon)^2-1&=\|v_{\epsilon}  \|_{L^2(B)}^2-1=\int_Bv_{\epsilon}  ^2dx-\int_{\Omega_{\epsilon}}u_{\epsilon}^2dx\\
&=\int_B(v_{\epsilon}^2-\tilde{u}_{\epsilon}^2)dx-\int_{B\setminus\Omega_{\epsilon}}(v_{\epsilon}  ^2-\tilde{u}_{\epsilon} ^2)dx+R_2(\epsilon).
\end{align*}
Then by using \eqref{ineq_1}, \eqref{ineq_2}, we obtain
\begin{equation}\label{asterisco}
\left|\sum_{n=2}^{+\infty}a_n(\epsilon)^2-1\right|\leq K_5\epsilon^2+C\omega(\epsilon)\leq C_1\omega(\epsilon). 
\end{equation}
We may now write 
\begin{align*}
\lambda_2(\Omega_{\epsilon})&=\int_{\Omega_{\epsilon}}|D^2u_{\epsilon}|^2+\tau|Du_{\epsilon}|^2dx\\
&=\int_B|D^2v_{\epsilon}|^2+\tau|Dv_{\epsilon}|^2\,dx+\int_{\Omega_{\epsilon}\setminus B}|D^2u_{\epsilon}|^2+\tau|Du_{\epsilon}|^2dx\\
&\qquad-\int_{B\setminus\Omega_{\epsilon}}|D^2v_{\epsilon}|^2+\tau|Dv_{\epsilon}|^2\,dx\\
&=\sum_{n=2}^{+\infty}a_n(\epsilon)^2\lambda_n(B)-R_1(\epsilon). 
\end{align*}
From Lemma \ref{lemma_refinement}, it follows that 
\[
|\lambda_2(B)-\lambda_2(\Omega_{\epsilon})|\leq C'\omega(\epsilon), 
\]
and therefore,
\begin{equation}\label{star1}
\left|\sum_{n=2}^{+\infty}a_n(\epsilon)^2\lambda_n(B)-\lambda_2(B)\right|=|\lambda_2(\Omega_{\epsilon})+R_1(\epsilon)-\lambda_2(B)|\leq  C_2\omega(\epsilon).
\end{equation}

By the symmetry of the ball, the first nonzero eigenvalue $\lambda_2(B)$ has multiplicity $N$, and so $\lambda_2(B)=\lambda_3(B)=\cdots=\lambda_{N+1}(B)<\lambda_{N+2}(B)$.  Therefore,
\begin{align*}
C_2\omega(\epsilon)
&\geq\left|\sum_{n=2}^{N+1}a_n(\epsilon)^2\lambda_2(B)+\sum_{n=N+2}^{+\infty}a_n(\epsilon)^2\lambda_n(B)-\lambda_2(B)\right|\\
&=\left|\lambda_2(B)\left(\sum_{n=2}^{+\infty}a_n(\epsilon)^2-1\right)+\sum_{n=N+2}^{+\infty}a_n(\epsilon)^2\left(\lambda_n(B)-\lambda_2(B)\right)\right|\\
&\geq\left(\lambda_{N+2}(B)-\lambda_2(B)\right)\sum_{n=N+2}^{+\infty}a_n(\epsilon)^2-\lambda_2(B)C_1\omega(\epsilon),
\end{align*}
which yields
\begin{equation}\label{C3}
\sum_{n=N+2}^{+\infty}a_n(\epsilon)^2\leq C_3\omega(\epsilon),
\end{equation}
and hence by \eqref{asterisco},
\begin{equation}\label{C4}
\left|\sum_{n=2}^{N+1}a_n(\epsilon)^2-1\right|\leq C_4\omega(\epsilon).
\end{equation}
Revisiting \eqref{star1}, we see that
\begin{align*}
C_2\omega(\varepsilon)&\geq\left|\sum_{n=2}^{N+1}a_n(\epsilon)^2\lambda_2(B)+\sum_{n=N+2}^{+\infty}a_n(\epsilon)^2\lambda_n(B)-\lambda_2(B)\right|\\
&=\left|\lambda_2(B)\left(\sum_{n=2}^{N+1}a_n(\epsilon)^2-1\right)+\sum_{n=N+2}^{+\infty}a_n(\epsilon)^2\lambda_n(B)\right|\\
&\geq\lambda_2(B)\left(\sum_{n=2}^{N+1}a_n(\epsilon)^2-1\right)+\sum_{n=N+2}^{+\infty}a_n(\epsilon)^2\lambda_n(B),
\end{align*}
which, together with \eqref{C3} and \eqref{C4}, yields
\begin{equation}\label{star2}
\sum_{n=N+2}^{+\infty}a_n(\epsilon)^2\lambda_n(B)\leq C_2\omega(\epsilon)-\lambda_2(B)\left(\sum_{n=2}^{N+1}a_n(\epsilon)^2-1\right)\leq C_5\omega(\epsilon).
\end{equation}

Now set $\varphi:=\sum_{n=2}^{N+1}a_n(\epsilon)\xi_n$ and define the norm $\|\cdot\|_{H^2_{\tau}(B)}$ by
\[
\|h\|_{H^2_{\tau}(B)}^2:=\int_B |D^2h|^2+\tau|Dh|^2+h^2\,dx,\qquad \forall h\in H^2(B).
\]
This norm is equivalent to the standard $H^2(B)$-norm by coercivity of the bilinear form. 

We now estimate the quantity $\|v_{\epsilon}  -\varphi\|_{H^{2}_{\tau}(B)}$. We have
\begin{align*}
\|v_{\epsilon}  -\varphi\|^2_{H^2_{\tau}(B)}&=\int_B|D^2(v_{\epsilon}  -\varphi)|^2+\tau|D(v_{\epsilon}  -\varphi)|^2+(v_{\epsilon}  -\varphi)^2dx\\
&=\int_B\sum_{n=N+2}^{+\infty}a_n(\epsilon)^2(|D^2\xi_n|^2+\tau|D\xi_n|^2+\xi_n^2)dx\\
&=\sum_{n=N+2}^{+\infty}a_n(\epsilon)^2(1+\lambda_n(B))\leq C_6\omega(\epsilon),
\end{align*}
where the last inequality follows from \eqref{C3} and \eqref{star2}. Thus the function $v_{\epsilon}$ is $\sqrt{\omega(\epsilon)}$-close to $\varphi$ in the $H^2_{\tau}(B)$-norm. 

We want to bound the $C^3(\overline B)$-norm with the $H^2_{\tau}(B)$-norm. To do so, we use standard elliptic regularity estimates for the biharmonic operator. We have that, in $B\cap\Omega_{\epsilon}$,
\[
\Delta^2 v_{\epsilon}  -\tau\Delta v_{\epsilon}  =\Delta^2u_{\epsilon}  -\tau\Delta u_{\epsilon}  =\lambda_2(\Omega_{\epsilon})u_{\epsilon}  =\lambda_2(\Omega_{\epsilon})(v_{\epsilon}  +\delta).
\]
Recall that $\delta\in O(\epsilon)$ as $\epsilon\rightarrow 0$ from \eqref{delta_bound}. We set
\begin{equation*}
f_{\epsilon}:=\Delta^2v_{\epsilon}  -\tau\Delta v_{\epsilon}  .
\end{equation*}
Note that, in particular, $f_{\epsilon}=\lambda_2(\Omega_{\epsilon})(v_{\epsilon}  +\delta)$ on $B\cap\Omega_{\epsilon}$. Then defining the functions $g_{\epsilon}^{(1)}$ and $g_{\epsilon}^{(2)}$ on $\partial B$ by $g_{\epsilon}^{(1)}:=\frac{\partial^2v_{\epsilon}  }{\partial\nu^2}$ and $g_{\epsilon}^{(2)}:=\tau\frac{\partial v_{\epsilon}  }{\partial\nu}-{\rm div}_{\partial B}(D^2 v_{\epsilon}  \cdot\nu)-\frac{\partial\Delta v_{\epsilon}  }{\partial\nu}$, we see that the function $v_{\epsilon}$ uniquely solves the problem
\begin{equation*}
\begin{cases}
\Delta^2u-\tau\Delta u=f_{\epsilon}, & {\rm in}\ B,\\
\frac{\partial^2u}{\partial\nu^2}=g_{\epsilon}^{(1)}, & {\rm on}\ \partial B,\\
\tau\frac{\partial u}{\partial\nu}-{\rm div}_{\partial B}(D^2u\cdot\nu)-\frac{\partial\Delta u}{\partial\nu}=g_{\epsilon}^{(2)}, & {\rm on}\ \partial B,\\
\int_{B}udx=0.
\end{cases}
\end{equation*}
 Now let $f:=\lambda_2(B)\varphi$. Then by definition, the function $\varphi$ is the unique solution of
\begin{equation*}
\begin{cases}
\Delta^2 u-\tau\Delta u=f, & {\rm in}\ B,\\
\frac{\partial^2u}{\partial\nu^2}=0, & {\rm on}\ \partial B,\\
\tau\frac{\partial u}{\partial\nu}-{\rm div}_{\partial B}(D^2u\cdot\nu)-\frac{\partial\Delta u}{\partial\nu}=0, & {\rm on}\ \partial B,\\
\int_B u\,dx=0.
\end{cases}
\end{equation*}
Finally, define the function $w:=v_{\epsilon}  -\varphi$, which is the unique solution of
\begin{equation*}
\begin{cases}
\Delta^2 w-\tau\Delta w=f_{\epsilon}-f, & {\rm in}\ B,\\
\frac{\partial^2w}{\partial\nu^2}=g_{\epsilon}^{(1)}, & {\rm on}\ \partial B,\\
\tau\frac{\partial w}{\partial\nu}-{\rm div}_{\partial B}(D^2w\cdot\nu)-\frac{\partial\Delta w}{\partial\nu}=g_{\epsilon}^{(2)}, & {\rm on}\ \partial B,\\
\int_B w\,dx=0.
\end{cases}
\end{equation*}
For any $p>N$, we have (see e.g., \cite[Theorem 2.20]{gazzola})
\begin{equation}\label{gazzola_estimate}
\|w\|_{W^{4,p}(B)}\leq C\left(\|f_{\epsilon}-f\|_{L^p(B)}+\|g_{\epsilon}^{(1)}\|_{W^{2-\frac{1}{p},p}(\partial B)}+\|g_{\epsilon}^{(2)}\|_{W^{1-\frac{1}{p},p}(\partial B)}\right).
\end{equation}
We consider separately the three summands in the right-hand side of \eqref{gazzola_estimate}. We start from the first summand. Recall that for any $x\in B\cap \Omega_{\epsilon}$, we have (see \eqref{regularity_v})
\[
f_{\epsilon}(x)=\lambda_2(\Omega_{\epsilon})(v_{\epsilon}  (x)+\delta).
\]
Since $\delta\in O(\epsilon)$ and $\lambda_2(\Omega_\epsilon)$ is bounded from above and from below, we have that $f_{\epsilon}(x)=\lambda_2(\Omega_{\epsilon})v_{\epsilon}(x)+O(\epsilon)$, and thus, as $\epsilon\rightarrow 0$, for any $p>N$, we have  (cf.\ Lemma~\ref{lemma_refinement})
\begin{align}\label{gaz1}
\|f_{\epsilon}-f\|_{L^p(B)}&=\|\lambda_2(\Omega_{\epsilon})v_{\epsilon}  -\lambda_2(B)\varphi\|_{L^p(B)}+O(\epsilon)\\
&\leq |\lambda_2(\Omega_{\epsilon})-\lambda_2(B)|\|v_{\epsilon}  \|_{L^p(B)}+|\lambda_2(B)|\|v_{\epsilon}  -\varphi\|_{L^p(B)}+O(\epsilon)\nonumber\\
&\leq C_7\omega(\epsilon)+C_8\sqrt{\omega(\epsilon)}+O(\epsilon)\leq C_9\sqrt{\omega(\epsilon)}.\nonumber
\end{align}

Now we consider the second summand in the right-hand side of \eqref{gazzola_estimate}. Since $g_{\epsilon}^{(1)}=\frac{\partial^2v_{\epsilon}  }{\partial\nu^2}$ and $v_{\epsilon}$ is an extension of $u_{\epsilon}$, by the regularity of both $u_{\epsilon}$ and $v_{\epsilon}$ (see \eqref{regularity}, \eqref{regularity_v}) and from the fact that $\frac{\partial^2u_{\epsilon} }{\partial\nu^2}=0$ on $\partial\Omega_{\epsilon}$, we may conclude
\begin{equation}\label{gaz2}
\|g_{\epsilon}^{(1)}\|_{W^{2-\frac{1}{p},p}(\partial B)}\leq C\epsilon.
\end{equation}
For the same reason, for the third summand in the right-hand side of \eqref{gazzola_estimate} we have
\begin{equation}\label{gaz3}
\|g_{\epsilon}^{(2)}\|_{W^{1-\frac{1}{p},p}(\partial B)}\leq C\epsilon.
\end{equation}

From \eqref{gazzola_estimate} and the bounds \eqref{gaz1}, \eqref{gaz2}, and \eqref{gaz3}, it follows that, for any $p>N$, 
\begin{equation*}
\|v_{\epsilon}  -\varphi\|_{W^{4,p}(B)}\leq C_{10}\sqrt{\omega(\epsilon)},
\end{equation*}
and thus, from the Sobolev embedding theorem, 
\begin{equation*}
\|v_{\epsilon}  -\varphi\|_{C^{3}(\overline B)}\leq \tilde C\sqrt{\omega(\epsilon)}.
\end{equation*}
The proof is concluded by setting $\xi_{\varepsilon}=\varphi$.
\end{proof}
\end{lem}

The next lemma gives us refined bounds on $|R_1(\epsilon)|$ and $|R_2(\epsilon)|$.
\begin{lem}\label{lemma3}
Let $\omega(t), v_{\epsilon}$ be as in Lemma \ref{lemma_refinement}. Suppose that for all $\epsilon>0$ small enough there exists an eigenfunction $\xi_{\epsilon}$ associated with $\lambda_2(B)$ such that
\begin{equation}\label{hypo}
\|v_{\epsilon}  -\xi_{\epsilon}\|_{C^3(\overline B)}\leq C \sqrt{\omega(\epsilon)},
\end{equation}
for some $C>0$ which does not depend on $\epsilon>0$. Then there exists $\tilde C>0$ which does not depend on $\epsilon$ such that $|R_1(\epsilon)|,|R_2(\epsilon)|\leq\tilde C\epsilon\sqrt{\omega(\epsilon)}$.
\begin{proof}
It is convenient to use spherical coordinates $(r,\theta)\in\mathbb{R}_{+}\times\mathbb{S}^{N-1}$ in $\mathbb R^N$ and the corresponding change of variables $x=\phi(r,\theta)$. 
We denote by $\mathcal D$ and $\tilde{\mathcal D}$ the sets $\mathcal D:=\partial(\Omega_{\epsilon}\setminus B)\cap\partial B$ and $\tilde{\mathcal D}=\partial(B\setminus \Omega_{\epsilon})\cap\partial B$. Observe that $\psi\geq0$ on $\mathcal D$ and $\psi\le0$ on $\tilde{\mathcal D}$. 
 
Thanks to the regularity of $u_{\epsilon}$ and $\tilde{u}_{\epsilon}$ by \eqref{regularity_extension}, on $\Omega_{\epsilon}\setminus B$ we have
\begin{eqnarray*}
D^2u_{\epsilon}  \circ\phi(1+\epsilon\psi,\theta)&=&D^2u_{\epsilon}  \circ\phi(1,\theta)+O(\epsilon),\\
Du_{\epsilon}  \circ\phi(1+\epsilon\psi,\theta)&=&Du_{\epsilon}  \circ\phi(1,\theta)+O(\epsilon),
\end{eqnarray*}
as $\epsilon\rightarrow 0$. Therefore, integrating with respect to the radius $r$ and applying the definition of $v_{\epsilon}$ \eqref{test}, we see
\begin{align*}
\int_{\Omega_{\epsilon}\setminus B}|D^2u_{\epsilon}|^2+\tau|Du_{\epsilon}|^2dx
&=\epsilon\int_{\mathcal D }\psi\left(\left|D^2u_{\epsilon}  \right|^2+\tau \left|Du_{\epsilon}  \right|^2\right)d\sigma+O(\epsilon^2)\\
&=\epsilon\int_{\mathcal D }\psi\left(\left|D^2v_{\epsilon}  \right|^2+\tau \left|Dv_{\epsilon}  \right|^2\right)d\sigma+O(\epsilon^2),
\end{align*}
as $\epsilon\rightarrow 0$. Similarly,
\begin{equation*}
\int_{B\setminus\Omega_{\epsilon}}|D^2v_{\epsilon}|^2+\tau|Dv_{\epsilon}|^2dx
=-\epsilon\int_{\tilde{\mathcal D}}\psi\left(\left|D^2v_{\epsilon}  \right|^2+\tau \left|Dv_{\epsilon}  \right|^2\right)d\sigma+O(\epsilon^2),
\end{equation*}
as $\epsilon\rightarrow 0$. From these and hypothesis \eqref{hypo}, we see
\begin{align}\label{ineqR1}
|R_1(\epsilon)|&\leq \epsilon\left|\int_{\partial B}\psi\left(\left|D^2v_{\epsilon}  \right|^2+\tau \left|Dv_{\epsilon}  \right|^2\right)d\sigma\right|+O(\epsilon^2)\\
&\leq \epsilon\left|\int_{\partial B}\psi\left(\left|D^2\xi_{\epsilon}\right|^2+\tau \left|D\xi_{\epsilon}\right|^2\right)d\sigma\right|+C\epsilon\sqrt{\omega(\epsilon)}+O(\epsilon^2)\nonumber\\
&\le \tilde{C}\epsilon\sqrt{\omega(\epsilon)},\nonumber
\end{align}
as $\epsilon\rightarrow 0$. In the last inequality, we have used the following identity for eigenfunctions of $\lambda_2(B)$:
\begin{equation}\label{spherical_harmonic}
\left.\left(\left|D^2\xi_{\epsilon}\right|^2+\tau \left|D\xi_{\epsilon}\right|^2\right)\right|_{\partial B}=(a\cdot x)^2
\end{equation}
for some $a\in\mathbb R^N$ (cf.\ \eqref{uk}).

By following the same scheme, we can prove the analogue of \eqref{ineqR1} for $R_2(\epsilon)$. This concludes the proof.
\end{proof}
\end{lem}
We can now proceed to complete the proof of Theorem \ref{theorem2}. 

Let $\omega_0(\epsilon):=|R_1(\epsilon)|+|R_2(\epsilon)|$. This function is continuous in $\epsilon$ and, moreover, has the property
\[
\frac{\epsilon^2}{K}\leq\omega_0(\epsilon)\leq K\epsilon.
\]
The first inequality follows from Theorem \ref{NeumannQI}, while the latter follows from the fact that $R_1,R_2\in O(\epsilon)$. By Lemma \ref{lemma2}, it follows that there exists an eigenfunction $\xi_{\epsilon}$ of the Neumann problem \eqref{NeumannPDE} on $B$ associated with eigenvalue $\lambda_2(B)$ such that
\[
\|v_{\epsilon}  -\xi_{\epsilon}\|_{C^3(\overline{B})}\leq C\sqrt{\omega_0(\epsilon)}. 
\]
Now we apply Lemma \ref{lemma3}, obtaining
\[
\omega_0(\epsilon)\leq2\tilde C\epsilon\sqrt{\omega_0(\epsilon)},
\]
and therefore
\[
\sqrt{\omega_0(\epsilon)}=\frac{|R_1(\epsilon)|+|R_2(\epsilon)|}{\sqrt{\omega_0(\epsilon)}}\leq 2\tilde C\epsilon.
\]
From this, it follows that $\omega_0(\epsilon)\leq 4\tilde C^2\epsilon^2$, and hence both $|R_1(\epsilon)|,|R_2(\epsilon)|\leq 4\tilde C^2\epsilon^2 $. Finally, we apply Lemma \ref{lemma_refinement} and obtain
\[
\lambda_2(B)\leq\lambda_2(\Omega_{\epsilon})+\mathcal C\epsilon^2
\]
for a constant $\mathcal C>0$ independent of $\epsilon\in(0,\epsilon_0)$. This concludes the proof of Theorem \ref{theorem2}.

\begin{rmk}\label{directions}
In \cite{brascopratelli}, the authors provided an explicit construction of a family $\{\Omega_{\epsilon}\}_{\epsilon}$ in $\mathbb R^2$ suitable for proving the sharpness of their {quantitative isoperimetric inequality for the fundamental tone of the Neumann Laplacian}. On the other hand, in \cite{brascosteklov}, the authors gave only sufficient conditions to generate the family $\{\Omega_{\epsilon}\}_{\epsilon}$, which are exactly those we apply in \eqref{perturbation}. We observe that the first two conditions, namely
\begin{equation}\label{ellipses}
\int_{\partial B}\psi d\sigma=\int_{\partial B}(a\cdot x)\psi d\sigma=0,
\end{equation}
have a purely geometrical meaning, and are used to prove inequalities \eqref{sharp1} and \eqref{sharp2} (cf.\ \cite[Lemma 6.2]{brascosteklov}). The latter has a stricter relation with the problem, since any function $\xi$ belonging to the eigenspace associated with $\lambda_2(B)$ satisfies equality \eqref{spherical_harmonic}. This is due to the fact that $\xi$ can be expressed as a radial part times a spherical harmonic polynomial of degree $1$. This also tells us that the correct conditions to impose when considering the Steklov problem are still \eqref{perturbation}. In particular, as pointed out in \cite[Remark 6.9]{brascosteklov}, ellipsoids satisfy conditions \eqref{ellipses}, and hence inequalities \eqref{sharp1} and \eqref{sharp2} hold, but miss the final condition, and therefore are not a suitable family for this problem. Note that for the Dirichlet Laplacian case in \cite{brasco2015}, ellipsoids are a suitable family for proving the sharpness, and therefore conditions \eqref{ellipses} are sufficient.

We also observe that in \cite{brasco2015}, the construction is somewhat more general (cf.\ \cite[Theorem 3.3, pp.\ 1788-1789]{brasco2015}), while the perturbation used in \cite{brascopratelli} does not belong to \eqref{perturbation}. This means that it is possible to state less-restrictive conditions which would produce families of domains achieving the sharpness.
\end{rmk}

\section{Sharpness of the Steklov inequality}\label{sharpness_steklov}

In this section, we prove the sharpness of inequality \eqref{quantitative_bp}. Due to the strong similarities between the Steklov problem \eqref{SteklovPDE} and the Neumann problem \eqref{NeumannPDE}, we shall maintain the same notation as in the previous section.

\begin{thm}
Let $B$ be the unit ball in $\mathbb{R}^N$ centered at zero. There exist a family $\left\{\Omega_{\epsilon}\right\}_{\epsilon>0}$ of smooth domains and positive constants $c_1,c_2,c_3,c_4$ and $r_1, r_2,r_3,r_4$ independent of $\epsilon>0$ such that
\begin{equation*}
r_1\epsilon^2\le\Big||\Omega_{\epsilon}|-|B|\Big|\leq r_2\epsilon^2,
\end{equation*}
\begin{equation*}
c_1\epsilon\leq c_2\mathcal A(\Omega_{\epsilon})\leq\frac{|\Omega_{\epsilon}\triangle B|}{|\Omega_{\epsilon}|}\leq c_3\mathcal A(\Omega_{\epsilon})\leq c_4\epsilon,
\end{equation*}
and
\begin{equation}
\label{sharp4}
r_3\epsilon^2\le\left|\lambda_2(\Omega_{\epsilon})-\lambda_2(B)\right|\leq r_4\epsilon^2,
\end{equation}
for all $\epsilon\in(0,\epsilon_0)$, where $\epsilon_0>0$ is sufficiently small, {and $\lambda_2(\Omega_\epsilon)$, $\lambda_2(B)$ is the first positive eigenvalue of problem \eqref{SteklovPDE} on $\Omega_\epsilon$, $B$ respectively}.
\end{thm}

To prove this theorem, we begin by defining the family $\left\{\Omega_{\epsilon} \right\}_{\epsilon>0}$ as in \eqref{family}. Thus it remains only to prove \eqref{sharp4}.

We remind the reader of the variational characterization of the first positive eigenvalue of the Steklov problem \eqref{SteklovPDE} on a domain $\Omega$:
\begin{equation}\label{steklov-ray}
\lambda_2(\Omega)=\inf_{\substack{0\ne u\in H^2(\Omega)\\ \int_{\partial\Omega}u \,d\sigma=0}}\frac{\int_{\Omega}|D^2u|^2+\tau|Du|^2\,dx}{\int_{\partial\Omega}u^2\,d\sigma}.
\end{equation}

We take the first positive eigenvalue $\lambda_2(\Omega_{\epsilon} )$ of the Steklov problem \eqref{SteklovPDE} on $\Omega_{\epsilon} $, and let $u_{\epsilon} $ be an associated eigenfunction, normalized by
\[
\int_{\partial\Omega_{\epsilon} }u_{\epsilon} ^2dx=1.
\]
Then by the variational characterization \eqref{steklov-ray},
\[
\int_{\Omega_{\epsilon} }|D^2u_{\epsilon} |^2+\tau|\nabla u_{\epsilon} |^2 dx=\lambda_2(\Omega_{\epsilon} ).
\]
By standard elliptic regularity (see e.g., \cite[\S 2.4.3]{gazzola}), since $\Omega_{\epsilon} $ is of class $C^{\infty}$ by construction, we have that $u_{\epsilon} \in C^{\infty}(\overline{\Omega_{\epsilon}})$ for all $\epsilon\in(0,\epsilon_0)$. Moreover, for all $k\in\mathbb N$, the sets $\Omega_{\epsilon} $ are of class $C^k$ uniformly in $\epsilon\in(0,\epsilon_0)$, which means that there exist constants $H_k>0$ independent of $\epsilon$ such that
\[
\|u_{\epsilon} \|_{C^k(\overline{\Omega_{\epsilon}} )}\leq H_k.
\] 
Let now $\tilde{u}_{\epsilon} $ be a $C^4$ extension of $u_{\epsilon} $ to an open neighborhood $A$ of $B\cup\Omega_{\epsilon} $. Then, there exists $K_A>0$ independent of $\epsilon>0$ such that
\[
\|\tilde{u}_{\epsilon} \|_{C^4(\overline A)}\leq K_A\|u_{\epsilon} \|_{C^4(\overline{\Omega_{\epsilon}} )}\leq K_A H_4.
\]
Analogous to the Neumann case, take $\delta:=\frac{1}{|\partial B|}\int_{\partial B}\tilde{u}_{\epsilon} \,d\sigma$ to be the mean of $\tilde{u}_{\epsilon}$ over $\partial B$. From the fact that $\int_{\partial\Omega_{\epsilon} }u_{\epsilon} dx=0$ and $|B\setminus\Omega_{\epsilon} |,|\Omega_{\epsilon} \setminus B|\in O(\epsilon)$ as $\epsilon\rightarrow 0$, it follows that, as $\epsilon\rightarrow0$ (see also \cite[formula (6.15)]{brascosteklov}),
\[
\delta=\frac{1}{|\partial B|}\int_{\partial B}\tilde{u}_{\epsilon} \,d\sigma\in O(\epsilon).
\]
Now let us set $v_{\epsilon}:=\tilde u_{\epsilon|_B}-\delta$. This function is of class $C^4(\overline B)$, satisfies $\int_{\partial B} v_{\epsilon} \,d\sigma=0$, and 
\[
\|v_{\epsilon} \|_{C^4(\overline B)}\leq K'
\]
for a constant $K'>0$ independent of $\epsilon\in(0,\epsilon_0)$. Therefore, $v_{\epsilon} $ is a suitable trial function for the Rayleigh quotient of $\lambda_2(B)$, hence,
\[
\lambda_2(B)\leq\frac{\int_B |D^2 v_{\epsilon} |^2+\tau|\nabla v_{\epsilon} |^2 dx}{\int_{\partial B} {v_{\epsilon} }^2 \,d\sigma}.
\]
On the other hand,
\[
\left|\int_{\partial B} v_{\epsilon} ^2-\tilde{u}_{\epsilon} ^2 \,d\sigma\right|=\left|\int_{\partial B} \delta^2-2\delta\tilde{u}_{\epsilon} \,d\sigma\right|\leq K''\epsilon^2,
\]
where $K''>0$ is a positive constant independent of $\epsilon\in(0,\epsilon_0)$.  Therefore, we may write
\[
\lambda_2(B)\leq\frac{\lambda_2(\Omega_{\epsilon} )+R_1(\epsilon)}{1+R_2(\epsilon)-\tilde K\epsilon^2},
\]
where we have once again defined the error terms
\[
R_1(\epsilon):=\int_{B\setminus\Omega_{\epsilon} }|D^2v_{\epsilon} |^2+\tau|\nabla v_{\epsilon} |^2dx-\int_{\Omega_{\epsilon} \setminus B}|D^2u_{\epsilon} |^2+\tau|\nabla u_{\epsilon} |^2 dx,
\]
and
\[
R_2(\epsilon):=\int_{\partial B}v_{\epsilon} ^2\,d\sigma-\int_{\partial \Omega_{\epsilon} }u_{\epsilon} ^2\,d\sigma.
\]

At this point, we note that the observations made in Section \ref{sharpness_neumann} remain valid here. Therefore, in order to conclude the proof of \eqref{sharp4}, we need only a few lemmas.

\begin{lem}
Let $\omega$ be as in Lemma \ref{lemma_refinement}. If there exists a constant $C>0$ such that $|R_1(\epsilon)|, |R_2(\epsilon)|\leq C\omega(\epsilon)$, then there exists a constant $C'>0$ such that
\[
\lambda_2(B)\leq\lambda_2(\Omega_{\epsilon} )+C'\omega(\epsilon)
\]
for every $\epsilon>0$ sufficiently small.
\end{lem}
\begin{proof}
See \cite[Lemma 6.7]{brascosteklov}.
\end{proof}

\begin{lem}
Let $\omega$ be as in Lemma \ref{lemma_refinement}. Suppose that there exists $C>0$ such that for all $\epsilon>0$ sufficiently small we have $|R_1(\epsilon)|, |R_2(\epsilon)|\leq C\omega(\epsilon)$. Then there exists an eigenfunction $\xi_{\epsilon}$ associated with $\lambda_2(B)$ such that
\[
\|v_{\epsilon} -\xi_{\epsilon}\|_{C^3(\overline B)}\leq\tilde C\sqrt{\omega(\epsilon)},
\]
for some $\tilde C>0$ independent of $\epsilon>0$.
\begin{proof}
The proof is essentially identical to that of Lemma \ref{lemma2} and hence the details are omitted. Some small changes are necessary since $L^2(\Omega)$-norms have to be replaced by $L^2(\partial\Omega)$-norms, since we are considering the Steklov problem.
\end{proof}
\end{lem}

\begin{lem}
Let $\omega$ be as in Lemma \ref{lemma_refinement}. Suppose that for all $\epsilon>0$ sufficiently small there exists an eigenfunction $\xi_{\epsilon}$ associated with $\lambda_2(B)$ such that
\[
\|v_{\epsilon} -\xi_{\epsilon}\|_{C^3(\overline B)}\leq C \sqrt{\omega(\epsilon)},
\]
for some $C>0$ independent of $\epsilon>0$. Then there exists $\tilde C>0$ independent of $\epsilon$ such that $|R_1(\epsilon)|,|R_2(\epsilon)|\leq\tilde C\epsilon\sqrt{\omega(\epsilon)}$.
\begin{proof}
Regarding the bound on $R_1$, we refer to the proof of Lemma \ref{lemma3}. For $R_2$, we refer to \cite[Lemma 6.8, p.\ 4701]{brascosteklov}, observing that if $\xi_{\epsilon}$ is an eigenfunction associated with $\lambda_2(B)$, then on $\partial B$, 
\[
{\rm div}_{\partial B}(D^2\xi_{\epsilon}\cdot\nu)+\frac{\partial\Delta\xi_{\epsilon}}{\partial\nu}=0,
\]
and therefore the second boundary condition in \eqref{SteklovPDE} reads as $\partial\xi_{\epsilon}/\partial\nu=\lambda_2(B)\xi_{\epsilon}$.
\end{proof}
\end{lem}

\section*{Acknowledgments}
The first and the third author wish to thank Berardo Ruffini for discussions on his paper \cite{brascosteklov}. The first author has been partially supported by the research project FIR (Futuro in Ricerca) 2013 `Geometrical and qualitative aspects of PDE's'.
The third author acknowledges financial support from the research project
`Singular perturbation problems for differential operators' Progetto di Ateneo
of the University of Padova and from the research project `INdAM GNAMPA Project 2015 - Un approccio funzionale analitico per problemi di perturbazione singolare e di omogeneizzazione'. 
The first and the third author are members of the Gruppo Nazionale
per l'Analisi Matematica, la Probabilit\`a e le loro Applicazioni (GNAMPA) of
the Istituto Naziona\-le di Alta Matematica (INdAM).

\bibliography{bibliography}{}
\bibliographystyle{abbrv}

\end{document}